\documentclass{elsarticle}
\usepackage{amsfonts, amsmath, amssymb}

\def\notdivides{\mathrel{\kern-3pt\not\!\kern3.5pt\bigm|}}
\def\smallnotdivides{\mathrel{\kern-2pt\not\!\kern3.5pt\vert}}

\newcommand{\fix}{\operatorname{\mathsf F}}

\newcommand{\ord}{\operatorname{ord}}

\newcommand{\laurent}{\mathbb{Z}[t^{\pm 1}]}

\renewcommand{\le}{\leqslant}
\renewcommand{\ge}{\geqslant}
\renewcommand{\epsilon}{\varepsilon}

\newtheorem{theorem}{Theorem}
\newtheorem{lemma}[theorem]{Lemma}
\newtheorem{proposition}[theorem]{Proposition}
\newtheorem{corollary}[theorem]{Corollary}
\newdefinition{definition}[theorem]{Definition}
\newdefinition{remark}[theorem]{Remark}
\newdefinition{example}[theorem]{Example}
\newproof{proof}{Proof}

\title{Towards a P{\'o}lya--Carlson dichotomy\\
for algebraic dynamics}
\author[jb]{Jason Bell}
\ead{jpbell@uwaterloo.ca}
\author[rm]{Richard Miles}
\author[tw]{Thomas Ward\corref{cor1}}
\ead{t.b.ward@durham.ac.uk}
\cortext[cor1]{Corresponding author.}
\address[jb]{Department of Pure Mathematics, University of Waterloo,
ON, CANADA N2L 3G1}
\address[rm]{School of Mathematics, University of East Anglia, NR4 7TJ, UK}
\address[tw]{Department of Mathematical Sciences, Durham University, DH1 3LE, UK}
\date{\today}
\begin{document}

\begin{abstract}
We present results and background rationale in support of
a P{\'o}lya--Carlson dichotomy between rationality and a
natural boundary for the analytic behaviour of dynamical zeta
functions of compact group automorphisms.
\end{abstract}
\maketitle
\section{Introduction}

Let~$\theta:X\to X$ be a continuous map on a compact
metric space with the property that
\[
\fix_\theta(n)=\vert\{x\in X\mid \theta^nx=x\}\vert
\]
is finite for all~$n\ge1$. The associated
dynamical zeta function
\[
\zeta_\theta(z)=\exp\sum_{n\ge1}\frac{\fix_\theta(n)}{n}z^n
\]
is an invariant of topological conjugacy for the map~$\theta$,
and the analytic properties of the zeta function and weighted
versions of it may be used to study orbit-growth and other
properties of~$\theta$. In particular, for situations in which
the zeta function has a finite positive radius of convergence
and a meromorphic extension beyond the radius of convergence
Tauberian methods may be used to relate analytic properties of
singularities of the zeta function to orbit-growth properties
of the map. For smooth maps with sufficiently uniform
hyperbolic behaviour, the zeta function is rational (see
Manning~\cite{MR0288786}), and in particular hyperbolic toral
automorphisms have this property. On the other hand, for some
natural families of dynamical systems the arithmetic and
analytic properties of the zeta function are known to be very
different. The third author~\cite{MR1702897}, for a family of
isometric extensions of the full shift on~$p$ symbols ($p$ a
prime) parametrised by a probability space, shows that with the
possible exception of two values of~$p$ the dynamical zeta
function is not an algebraic function almost surely. Everest,
Stangoe and the third author~\cite{MR2180241} studied the specific
automorphism of a compact group dual to the
automorphism~$r\mapsto2r$ on~$\mathbb{Z}[\frac16]$, showing it
to have a natural boundary on the circle~$\vert z\vert=\frac12$
(we refer to Segal~\cite[Ch.~6]{MR2376066} for a convenient
introduction to the theory of complex functions with natural
boundary). Buzzi~\cite{MR1925634} shows that a certain weighted
random zeta function has a natural boundary. In a different
direction, for dynamical systems with a polynomial growth bound
on the number of periodic orbits, a more natural complex
function that captures all the periodic point data is given by
an orbit Dirichlet series, and many natural examples are known
to have infinitely many singularities on the critical line with
no lower bound on their separation by work of Everest, Stevens
{\it et al.}~\cite{MR2550149}, or even to have a natural
boundary in calculations of Pakapongpun and the third
author~\cite{pw}.

One of the fundamental links between the arithmetic
properties of the coefficients of a complex power
series and its analytic behaviour is given
by the P{\'o}lya--Carlson theorem~\cite{MR1544479},~\cite{MR1512473}.

\medskip{\noindent\bf P\'{o}lya--Carlson Theorem.}
{\it A power series with integer coefficients
and radius of convergence~$1$ is either rational or has the
unit circle as a natural boundary.}\medskip

Unfortunately, while there are some natural group
automorphisms whose zeta function has radius of convergence~$1$,
many do not -- and for most (in cardinality) it is not
at all clear how to compute the radius of convergence
without refined information about the arithmetic
of linear recurrence sequences.

The suggestion we wish to explore here is that
there is a P{\'o}lya--Carlson dichotomy for
group automorphisms in exactly the same sense:
the zeta function of a compact group automorphism
is either rational or admits a natural boundary at
its radius of convergence. We cannot prove this statement,
but will show it for a large class of automorphisms
of connected finite-dimensional groups (these groups
are called solenoids). In addition, the
arguments here do we hope make this suggestion plausible,
and clarify what sort of issues would arise in attempting
to prove the full statement.

In addition to the P{\'o}lya--Carlson theorem, we will
make essential use of the Hadamard quotient theorem
(see van der Poorten~\cite{MR929097} and Rumely~\cite{MR990517}).

\medskip{\noindent\bf Hadamard Quotient Theorem.}
{\it Let~$\mathbb K$ be a field of characteristic zero, and suppose
that~$\sum_{n\ge0}b_nz^n$ and~$\sum_{n\ge1}c_nz^n$ in~$\mathbb{K}[[z]]$ are
expansions of rational functions.
If there is a finitely-generated ring~$R$ over~$\mathbb{Z}$
with~$a_n=\frac{b_n}{c_n}\in R$ for all~$n\ge1$, then~$\sum_{n\ge0}a_nz^n$ is
also
the expansion of a rational function.}\medskip

Given the fact that there are well-known arithmetical
constraints on the possible sequences~$(\fix_{\theta}(n))$ of
periodic point counts for any map (see Puri and the third
author~\cite{MR1873399}), and additional (less well-known)
constraints in the case of group automorphisms (see the thesis
of Moss~\cite{moss} for more details, or the survey of Staines {\it et
al.}~\cite{survey} for an example of linear recurrent divisibility
sequence that counts periodic points for some map but that
cannot be the periodic point count for a group automorphism), we should point out that the suggested
dichotomy certainly cannot hold for all maps. To see this,
notice for example that there is a continuous map~$\theta$
with~$\fix_{\theta}(n)=\binom{2n}{n}$ for all~$n\ge1$ (by work
of Puri and the third author~\cite{MR1873399}) and so
\[
\sum_{n\ge1}\fix_{\theta}(n)z^n=\frac{1}{\sqrt{1-4z}}-1.
\]

In the parameter space of all compact abelian group automorphisms over which we are suggesting the
dichotomy holds, the results below are restricted in three
different ways. Removing the assumption of connectedness --- or
at least dealing with the totally disconnected case --- is
likely to be relatively straightforward because of the
additional information available about the arithmetic properties
of linear recurrence sequences over finite fields, and it is expected
that the most interesting arithmetic questions occur in the
connected setting considered here. The bound on the topological
dimension of the compact group is needed to avoid Salem
numbers, the familiar bane of several investigations in
arithmetic and dynamics. Removing this bound on the face of it
would involve subtle Diophantine problems involving
linear forms in logarithms. The third way in
which we restrict the cases we deal with concerns a subset~$S$
of a countable collection~$P$ of places of a number field. We
are able to handle the situations in which~$S$ is finite or
infinite but extremely thin, and the case in which~$P\setminus
S$ is finite. Understanding the general case seems to require
different techniques.

\section{One-dimensional solenoids}

A one-dimensional solenoid~$X$ has a Pontryagin dual group
isomorphic to a subgroup of~$\mathbb{Q}$. Given an
automorphism~$\theta:X\rightarrow X$ of a one-dimensional solenoid,
the dual group naturally carries the structure of a
module over a ring of the form~$\mathbb{Z}[r^{\pm 1}]$,
where~$r\in\mathbb{Q}^\times$ and multiplication by~$r$
corresponds to application of the dual
automorphism~$\widehat{\theta}$. To avoid trivial (from a
dynamical point of view, non-ergodic) automorphisms, we
assume~$r\neq\pm 1$ throughout. For a more detailed account of
these systems, we refer to the papers of Chothi, Everest
and the third author~\cite{MR1461206} and the recent survey by Staines {\it
et al.}~\cite{survey}. Relevant background and references for
all the results we will need on linear recurrence sequences may
be found in the monograph of Everest, van der Poorten,
Shparlinski and the third author~\cite{MR1990179}.

There is a convenient formula for~$\fix_\theta(n)$, which we
shall use as the basis for our discussion. Since there may be
uncountably many non-conjugate automorphisms of
one-dimensional solenoids that share the same zeta function (we
refer to Miles~\cite[Ex.~1]{MR2441142} for an example), this
not only avoids complex classification problems for these
systems but has the
added advantage that no background in algebraic dynamics is
needed to access the main results of this section; instead the
formula~\eqref{one_solenoid_pp_formula} below serves as a
starting point.

Let~$\mathcal{P}(\mathbb{Q})$ denote the set of rational
primes. For any~$x\in\mathbb{Q}$
and~$S\subset\mathcal{P}(\mathbb{Q})$, write~$|x|_S=\prod_{p\in
S}|x|_p$. Miles~\cite[Th.~3.1]{MR2441142} shows that there is a
distinguished set of primes~$T\subset\mathcal{P}(\mathbb{Q})$
such that~$\fix_\theta(n)=|r^n-1|_T^{-1}$, for all~$n\ge 1$,
and that~$|r|_p=1$ for all~$p\in T$ (see
Miles~\cite[Rmk.~1]{MR2441142}). Therefore, by the Artin
product formula (see Weil~\cite[Sec.~IV.4]{MR0234930} for a
complete treatment of the valuation theory of number fields
used here) we have
\begin{equation}\label{one_solenoid_pp_formula}
\fix_\theta(n)=|r^n-1|\cdot|r^n-1|_S,
\end{equation}
where~$S=\mathcal{P}(\mathbb{Q})\setminus T$.
Furthermore,~$|r|_p\neq 1$ necessarily implies that~$p\in S$.

If~$T=\varnothing$, then we obtain the trivial
sequence~$(1,1,\dots)$, so we assume that~$T\neq\varnothing$
throughout. In this section, we do not discuss further the
problem of determining which dynamical systems give rise to the
same formula~\eqref{one_solenoid_pp_formula}, it is sufficient
to note that for any~$r\neq\pm 1$ and
any~$S\subset\mathcal{P}(\mathbb{Q})$ with~$|r|_p=1$ for
all~$p\in\mathcal{P}(\mathbb{Q})\setminus S$, the
automorphism~$x\mapsto rx$ on the
ring~$R=\mathbb{Z}[\frac{1}{p}:p\in S]$ dualizes to an
automorphism~$\theta$ of the solenoid~$X=\widehat{R}$, and its
sequence of periodic point counts is given
by~\eqref{one_solenoid_pp_formula}. We refer
to~\cite{MR1461206} for further details of this construction.

From now on, we assume that~$r\in\mathbb{Q}\setminus\{\pm1\}$
is fixed and consider the
sequences~$(f_S)$
defined by~$f_S(n)=|r^n-1|\cdot|r^n-1|_S$ for~$n\ge1$
that arise by varying the set~$S$.
We also write~$F_S(z)=\sum_{n\ge1}f_S(n)z^n$ for the
associated ordinary generating function.
We are able to concentrate on the
ordinary generating function
rather than the
zeta function
(and hence
have
more ready
access to the theory of linear recurrence sequences)
because of the following fundamental relationship.

\begin{lemma}\label{zeta_generating_function_link}
Let~$F(z)=\sum_{n\geqslant 1}\fix_\theta(n)z^n$.
If~$\zeta_\theta$ is rational then~$F$ is rational.
If~$\zeta_\theta$ has analytic continuation beyond
its circle of convergence, then so too does~$F$.
In particular, the existence of a natural boundary
at the circle of convergence for~$F$ implies the
existence of a natural boundary for~$\zeta_\theta$.
\end{lemma}

\begin{proof}
This follows from the fact that~$F(z)=z\zeta_\theta'(z)/\zeta_\theta(z)$.
\qed\end{proof}

In order to handle the sequence~$f_S=(f_S(n))$ more easily, we need a
way to evaluate expressions of the form~$|r^n-1|_p$
when~$|r|_p=1$. To this end, the following lemma is useful, and
we state this in the more general setting of number fields,
since this will be needed in the next section. Moreover, it
will be necessary to deal with number fields of unknown degree for
the putative linear recurrence sequences that arise
inside arguments by contradiction. Let~$\mathbb{K}$ be an
algebraic number field, and let~$\mathcal{P}(\mathbb{K})$
denote the set of places of~$\mathbb{K}$. For a
place~$v\in\mathcal{P}(\mathbb{K})$ with~$|\xi|_v=1$,
let~$\mathfrak{K}_v$ denote the residue class field, let~$m_v$
denote the multiplicative order of the image of~$\xi$
in~$\mathfrak{K}_v^\times$, and let~$\varrho_v$ denote the
residue degree. We assume that~$|\cdot|_v$ is normalized so
that the Artin product formula holds (we refer to Ramakrishnan
and Valenza~\cite{MR1680912} or
Weil~\cite[Sec.~IV.4]{MR0234930} for the details).

\begin{lemma}\label{fundamental_evaluation_lemma}
Let~$p$ be the characteristic of~$\mathfrak{K}_v$. There exists
a non-negative integer constant~$D_v\geqslant 0$ and a rational
constant~$C_v>0$, such that for any~$n\in\mathbb{N}$,
\[
|\xi^n-1|_v = \left\{
\begin{array}{ll}
1 & \textrm{ if } m_v\nmid n, \\
C_v |n|_p^{\varrho_v} & \textrm{ if } m_v\mid n\mbox{ and }
 \ord_p(n)> D_v,
\end{array}
\right.
\]
and~$|\xi^n-1|_v$ assumes at most finitely many values otherwise.
\end{lemma}

\begin{proof}
See~\cite[Lem.~4.9]{MR2308145}, for example.
\qed\end{proof}

\begin{remark}\label{fundamental_lemma_extension}
The proof of~\cite[Lem.~4.9]{MR2308145} shows in fact
that~$D_v$ is the least positive integer~$D$ such that
\[
|\xi^{m_p p^{D+1}}-1|_v=|p|_v|\xi^{m_p p^{D}}-1|_v,
\]
and also that
\[
C_v=|p|_v^{-D_v}|\xi^{m_p p^{D_v}}-1|_v.
\]
Furthermore, in the particular
case~$\mathbb{K}=\mathbb{Q}$,~$D_p=0$ when~$p>2$,
and~$D_p\in\{0, 1\}$ when~$p=2$.
\end{remark}

We begin by establishing precisely when~$F_S$ is rational.

\begin{theorem}\label{rationality_theorem}
The function~$F_S$ is rational if and only
if~$|r|_p\neq 1$ for all~$p\in S$.
\end{theorem}

\begin{proof}
Let~$S'=\{p\in S:|r|_p=1\}$
and~$S''=\{p\in S:|r|_p> 1\}$. Then
\[
f_S(n)=|r^n-1||r|_{S''}^n f(n),
\]
where~$f(n)=|r^n-1|_{S'}$. If~$r=a/b$,
then without loss of generality we can
assume that~$a>|b|$ because~$\zeta_{\theta}=\zeta_{\theta^{-1}}$.
Then
\begin{equation}\label{im_just_a_dirty_dog}
|r^n-1||r|_{S''}^n=a^n-b^n
\Rightarrow
f_S(n)=(a^n-b^n)f(n),
\end{equation}
and this shows immediately that~$F_S$ is rational if~$S'=\varnothing$.

To prove the converse, we assume that~$S'\neq\varnothing$
and aim to show that
the sequence~$f_S$ does not satisfy a linear recurrence
over~$\mathbb{Q}$, which will in turn imply that~$F_S$ is irrational.

To begin with, assume that~$S$ is finite, so
the sequence~$f=(f(n))$ lies in a finitely generated
extension of~$\mathbb{Z}$. For a contradiction,
assume that~$f_S$ is given by a linear recurrence relation.
Then, using (\ref{im_just_a_dirty_dog}) and the Hadamard quotient theorem, it follows that~$f$ also satisfies a
linear recurrence relation.
Let~$q$ be a rational prime not in~$S$, and define
\[
n(e)=q^e\prod_{p\in S'}m_p p^{D_p},
\]
where~$D_p$ is as in Lemma~\ref{fundamental_evaluation_lemma}
and~$e\geqslant 1$. Applying
Lemma~\ref{fundamental_evaluation_lemma}, we see
that
\[
f(kn(e))=f(n(e))
\]
whenever~$k$ is coprime to~$n(e)$. Hence the sequence~$f$
assumes infinitely many values infinitely often, and so it
cannot satisfy a linear recurrence
by a result of Myerson and van der
Poorten~\cite[Prop.~2]{MR1357486}, giving a contradiction.

Now assume that~$S$ is infinite and
recall that~$S\neq\mathcal{P}(\mathbb{Q})$.
Using the Artin product formula and~(\ref{im_just_a_dirty_dog}), if $T=\mathcal{P}(\mathbb{Q})\setminus S'$, then \[
f_S(n)f_T(n)=|r^n-1||r^n-1|_{S''}=|r^n-1||r|^n_{S''}=a^n-b^n.
\]
Note that both $f_S$ and $f_T$ are positive integer sequences and the product sequence $f_Sf_T$ satisfies a linear recurrence over~$\mathbb{Q}$. Moreover, by the Hadamard quotient theorem, $f_S$ satisfies a linear recurrence over~$\mathbb{Q}$ if and only if~$f_T$
satisfies a linear recurrence over~$\mathbb{Q}$.

Hence, if either $f_S$ or $f_T$ satisfies a linear recurrence, it follows that for~$n$ sufficiently large we have
\begin{equation}
\label{eq: f}
f_S(n) \ = \ \sum_{i=1}^d P_i(n) \alpha_i^n\end{equation}
 and
\begin{equation}
\label{eq: g}
f_T(n) \ = \ \sum_{j=1}^e Q_j(n) \beta_j^n,
\end{equation}
for some natural numbers~$d$ and~$e$ and
rational polynomials
\[
P_1,\ldots,P_d, Q_1,\ldots ,Q_e,
\]
and algebraic
numbers~$\alpha_1,\ldots ,\alpha_d,\beta_1,\ldots ,\beta_e$.
Let~$N$ be a nonzero natural number with the property
that~$NP_i$ and~$NQ_j$ are integer polynomials
and~$N\alpha_i$ and~$N\beta_j$ are algebraic integers
for~$i,j\in \{1,\ldots ,d\}\times \{1,\ldots ,e\}$.
We let~$N'$ denote the product of the norms
of~$\alpha_1,\ldots ,\alpha_d$ and~$\beta_1,\ldots ,\beta_e$.

Let~$p$ be a prime number that does not divide~$abNN'$. We may
then regard~$\alpha_1,\ldots ,\alpha_d,\beta_1,\ldots ,\beta_e$
as elements of the algebraic closure~$\overline{\mathbb{Q}}_p$,
and since~$p$ does not divide~$N$, we have that
\[
|\alpha_i|_p , |\beta_j|\le 1
\]
for~$(i,j)\in \{1,\ldots ,d\}\times \{1,\ldots ,e\}$. Moreover,
since~$p$ does not divide~$N'$, we
have
\[
|\alpha_i|_p=|\beta_j|_p=1
\]
for~$(i,j)\in \{1,\ldots ,d\}\times \{1,\ldots ,e\}$.
Let~$G\in \mathbb{Q}[x]$ be the monic polynomial of smallest
degree that has~$\alpha_1,\ldots ,\alpha_d,\beta_1,\ldots
,\beta_e$ as zeros. Then, by construction,~$G$ has no
coefficients with denominators divisible by~$p$ when written in
lowest terms. We note that~$G$ modulo~$p$ splits in some
extension~$\mathbb{F}_q$ of~$\mathbb{F}_p$, so there is
some~$k\ge 0$ such that
\[
|\alpha_i^{p^k}-\alpha_i|_p<1
\]
and
\[
|\beta_j^{p^k}-\beta_j|_p<1
\]
for~$(i,j)\in \{1,\ldots ,d\}\times\{1,\ldots ,e\}$.
In particular, we have~$|\alpha_i^{p^k-1}-1|_p<1$
for~$i\in \{1,\ldots ,d\}$
and~$|\beta_j^{p^k-1}-1|_p<1$ for~$j\in \{1,\ldots ,e\}$.
We also have
\[
|P_i(p^k-1)-P_i(-1)|_p<1
\]
and
\[
|Q_j(p^k-1)-Q_j(-1)|_p<1
\]
for~$(i,j)\in \{1,\ldots ,d\}\times\{1,\ldots ,e\}$
(simply because~$P_i$ and~$Q_j$ are
integral polynomials).
It follows from~\eqref{eq: f} and~\eqref{eq: g}
that~$p$ divides~$f_S(p^k-1)$ if and only
if~$p$ divides~$\sum_{i=1}^d P_i(-1)$ and~$p$
divides~$f_T(n)$ if and only if~$p$ divides~$\sum_{j=1}^e Q_i(-1)$.
Since~$f_S(p^k-1)f_T(p^k-1)=a^{p^k-1}-b^{p^k-1}\equiv 0~(\bmod \, p)$,
we see that
\[
\left(\sum_{i=1}^d P_i(-1)\right)\left( \sum_{j=1}^e Q_i(-1)\right)
\equiv 0~(\bmod \, p)
\]
for all sufficiently large primes~$p$.
Hence one of~$\sum_{i=1}^d P_i(-1)$ and~$\sum_{j=1}^e Q_j(-1)$
must be zero.
Without loss of generality, assume that~$\sum_{i=1}^d P_i(-1)=0$.
This means that~$p$ divides~$f_S(p^k-1)$ for every prime~$p$
that does not divide~$abNN'$.
In particular,~$p\not \in S$
for any such prime~$p$.

Thus we may assume that~$S$ is finite and non-empty,
and this case has been handled already.
\qed\end{proof}

We record two curious consequences for the special case of
one-dimensional solenoids.

\begin{corollary}
The function~$F_S$ is rational if and only if
the associated zeta function is rational.
\end{corollary}

\begin{proof}
The proof above shows that
if~$F_S$ is rational then~$f_S(n)=a^n-b^n$
with~$a>\vert b\vert\ge1$,
in which case~$\zeta_\theta(z)=\frac{1-bz}{1-az}$.
The other implication is covered by
Lemma~\ref{zeta_generating_function_link}.
\qed\end{proof}

\begin{corollary}\label{radiusonecorollary}
If~$F_S$ has radius of convergence~$1$,
then~$F_S$ has the unit circle as a natural boundary.
\end{corollary}

\begin{proof}
The theorem above shows that if~$F_S$ is rational,
then~$f_S(n)=a^n-b^n$, where~$a>|b|\geqslant 1$,
so~$F_S$ has radius of convergence~$\frac{1}{a}<1$.
Thus,~$F_S$ having radius of convergence~$1$
implies that~$F_S$ is irrational, and the result
follows from the P\'{o}lya--Carlson Theorem itself.
\qed\end{proof}

As an example of Corollary~\ref{radiusonecorollary}, consider the
case where~$\mathcal{P}(\mathbb{Q})\setminus S$ is finite
and non-empty, so~$f_S(n)$ grows polynomially in~$n$
(as~$f_S$ corresponds to a sequence of periodic point
counts for a system of finite combinatorial rank~\cite{MR2550149}),
and~$F_S$ has radius of convergence~$1$.
Then the corollary shows that~$F_S$ has a natural boundary.
Corollary~\ref{radiusonecorollary} also applies to
many of the cases where~$S$ is infinite, but it is not
straightforward to exibit examples.

We now turn our attention to finite~$S$, for
which the radius of convergence of~$F_S$ is always strictly less
than~$1$. It will be useful to consider
certain rational functions that arise from
congruence conditions, especially for our subsequent work
involving Lambert series. The following
(a
generating function analogue of the Euler product
construction
for certain Dirichlet series)
is readily
established using a simple application of the
inclusion-exclusion principle.

\begin{lemma}\label{useful_rational_functions}
For a finite set of rational primes~$S$,
the function~$H_S(z)=\sum_{n\ge1}z^n$,
where~$n$ runs over all positive integers
with~$p\nmid n$ for all~$p\in S$, is a
rational function of the form
\[
\sum_{I\in\mathcal{I}}\frac{d_Iz^{k_I}}{1-z^{k_I}},
\]
where~$\mathcal{I}=\mathcal{I}(S)$ is a finite indexing
set,~$d_I\in\{-1,1\}$, and
each~$k_I\in\mathbb{N}$ is divisible only by primes
appearing in~$S$.
\end{lemma}

For a singleton set, we write more
briefly~$H_{\{p\}}=H_p$. Thus, for example,
\[
H_3(z)=\displaystyle\frac{z}{1-z}-\frac{z^3}{1-z^3}.
\]

The next example is
a simplified account of
the case considered by Everest, Stangoe and
the third author~\cite{MR2180241}. It gives a simple
illustration of the irrational case in
Theorem~\ref{rationality_theorem}, and
the functional equation found
may be used to show the existence of a natural boundary in this
case.

\begin{example}\label{vicky_example}
For $r=2$, consider~$F_{\{2,3\}}$, which is the ordinary generating
function for the periodic point sequence for the map dual
to~$x\mapsto 2x$ on~$\mathbb{Z}[\frac16]$. For this simple
example, we can establish a functional equation to show
that~$F_{\{2,3\}}$ has a natural boundary. Let
\[
F(z)=\sum_{n\ge
1}|2^n-1|_3 z^n,
\]
so that~$F_{\{2,3\}}(z)=F(2z)-F(z)$.
Since~$F$ has radius of convergence~$1$, showing that the
unit circle is a natural boundary for~$F$ is enough to prove
that the circle~$|z|=\frac{1}{2}$ is a natural boundary
for~$F_{\{2,3\}}$. Since
\[
F(z)=
\frac{1}{3}\sum_{2\mid n}|n|_3z^n
+ \sum_{2\nmid n}z^n,
\]
we have~$F(z)=\frac{1}{3}G(z^2)+H_2(z)$,
where~$G(z)=\sum_{n\ge1}|n|_3z^n$. Furthermore,
since~$H_2$ is rational, it is enough to show
that~$G$ has the natural boundary~$|z|=1$ to establish
this for~$F$. Writing~$n=3^e k$,
where~$e\geqslant 0$ and~$3\nmid k$, gives
\[
G(z)=
\sum_{e\geqslant 0}\frac{1}{3^e}\sum_{3\nmid k}z^{3^e k}
=
\sum_{e\geqslant 0}\frac{1}{3^e}H_3(z^{3^e})
=
H_3(z)
+
\frac{1}{3}\sum_{e\geqslant 0}
\frac{1}{3^e}H_3(z^{3^{e+1}}).
\]
It follows that~$G(z)=H_3(z)+\frac{1}{3}G(z^3)$. Using this
functional equation inductively, we deduce that there are
dense singularites of~$G$ on the unit circle,
occuring at~$3^{e}$-th roots of unity,~$e\in\mathbb{N}$.
\end{example}

In general, it is difficult to establish functional
equations of the same sort to demonstrate a natural boundary.
However, for finite sets~$S$, we are nonetheless able to
identify distinguished singularites on the circle of
convergence for~$F_S$ that lead to a natural boundary,
by means of the following calculation.

\begin{theorem}\label{distinguished_singularities}
Let~$S$ be a finite set of rational primes
such that~$|r|_p=1$ for all~$p\in S$,
and let~$F(z)=\sum_{n\geqslant 1}|r^n-1|_S z^n$.
Then there is a constant~$E(S)>0$ such that for
any~$q\in S$ and any~$\delta\in\mathbb{Z}[1/q]$,
with the possible exception of finitely many
values of~$\delta$,
\[
\left|F(\lambda \exp(2\pi \delta i ))\right|\rightarrow\infty
\]
as~$\lambda\rightarrow 1^-$
whenever~$|\delta|_q>q^{E(S)}$.
\end{theorem}

\begin{proof}
Let~$m_p$ denote the multiplicative order of~$r$ modulo~$p$ for
any prime~$p$ in~$S$, and note that~$|r^n-1|_p\neq 1$ if and
only if~$m_p\mid n$ by
Lemma~\ref{fundamental_evaluation_lemma}. Let~$T$ comprise the
set of primes in~$S$ together with all those that divide~$m_p$
for some~$p\in S$. Choose~$\delta\in\mathbb{Z}[1/q]$ with
\[
-\ord_q(\delta)>E(S)=1+\max\{\ord_t(m_p):t\in T, p\in S\}
\]
and set~$E=-\ord_q(\delta)-1$. We wish to consider the
behaviour of~$F(z)$ when
\[
z=\lambda\exp(2\pi \delta i)
\]
and~$\lambda\rightarrow 1^-$. To do this, we will split up the
sum defining~$F$ as follows.

Let
\[
J=\left\{
\textstyle\prod_{p\in T} p^{e_p}:0\leqslant e_p< E\mbox{ for all }p\in T\right\}
\]
and for any~$j\in J$, let~$N(j)$ denote the set of positive
integers~$n$ such that
\[
\ord_p(n)=\ord_p(j)
\]
for all~$p\mid j$ and~$\ord_p(n)\geqslant E$ for all~$p\in T$
with~$p\nmid j$, so that~$\{N(j)\}_{j\in J}$ forms a partition
of~$\mathbb{N}$. Notice that
\[
m_p\mid n\mbox{ for some } n\in N(j)
\Longleftrightarrow
m_p\mid n\mbox{ for all } n\in N(j)
\]
for any~$p\in S$.
Furthermore, if we define
\[
S(j)=\{p\in S:p\nmid j\mbox{ and } m_p\mid n\mbox{ for some } n\in N(j)\},
\]
then by~Lemma~\ref{fundamental_evaluation_lemma}, for all~$n\in
N(j)$,
\[
|r^n-1|_S
=
c_j|n|_{S(j)}
\]
for some non-negative rational constant~$c_j$. Hence, we can write
\begin{eqnarray*}
F(z)
& = &
\sum_{j\in J}
\sum_{n\in N(j)}
c_j|n|_{S(j)}z^n\\
& = &
\underbrace{
\sum_{j:S(j)\neq\varnothing}
c_j
\sum_{n\in N(j)}
|n|_{S(j)}z^n
}_{G(z)}
+
\sum_{j:S(j)=\varnothing}
c_j
\sum_{n\in N(j)}
z^n.
\end{eqnarray*}
The second series on the right-hand side (given by summands for
which~$S(j)$ is empty) is a rational function with radius of
convergence~$1$ by Lemma~\ref{useful_rational_functions}, so
has only finitely many singularities on the circle of
convergence~$|z|=1$. Therefore, for all but finitely many
choices of~$\delta\in\mathbb{Z}[1/q]$, this second series is
bounded as~$\lambda\rightarrow 1^-$. From now on assume
that~$\delta$ has been chosen so as to avoid these
singularities. To demonstrate the required conclusion of the
theorem, it is sufficient to show
that~$|G(z)|\rightarrow\infty$ as~$\lambda\rightarrow 1^-$. For
ease of notation, we now omit the clause~$S(j)\neq\varnothing$
when writing~$G(z)$.

For any~$j\in J$, let
\[
U(j)=\left\{
\textstyle\prod_{p\in S(j)} p^{e_p}:e_p \geqslant E\mbox{ for all }p\in S(j)\right\},
\]
and
\[
j'=j\prod_{p\in T\setminus S(j):p\nmid j}p^E.
\]
Note that any~$n\in N(j)$ can be written uniquely in the
form~$n=uj'k$ with~$u$ in~$U(j)$ and~$p\nmid k$ for all~$p\in
S'(j)=S(j)\cup\{p:p\mid j\}$. Hence,
\begin{equation}\label{that_could_cost}
G(z)
 =
\sum_{j}
\sum_{u\in U(j)}
\frac{c_j}{u}
\sum_{k}
z^{uj'k},
\end{equation}
where~$k$ runs through all positive integers with~$p\nmid k$
for all~$p\in S'(j)$. Consider first the terms in the series
for which~$\ord_q(uj)>E$. In this
case,~$z^{uj'k}=\lambda^{uj'k}$ and the inner sum comprises
strictly positive real terms, and diverges
as~$\lambda\rightarrow1^-$. To complete the proof, we will show
that the sum of the remaining terms in the series for~$G(z)$ is
bounded. For these terms, we have~$\ord_q(uj)\leqslant E$.

First suppose~$q\nmid uj$, so~$q\not\in S(j)$ and~$q\nmid j$,
which implies that~$q\not\in S'(j)$. Using
Lemma~\ref{useful_rational_functions} for the corresponding
terms of~$G(z)$, the inner sum in~(\ref{that_could_cost}) may
be written as~$H_{S'(j)}(z^{uj'})$. Moreover, the rational
expression given in Lemma~\ref{useful_rational_functions} shows
that
\begin{equation}\label{somebodys_life}
|H_{S'(j)}(z^{uj'})|\leqslant
\sum_{I\in \mathcal{I}(S'(j))}
\frac{1}{|1-z^{uj'k_I}|},
\end{equation}
and since~$q\not\in S'(j)$,~$\ord_q(uj'k_I)=\ord_q(uj')=E$.
Thus there is a constant~$M(\delta,j)>0$ such
that
\[
|1-z^{uj'k_I}|>M(\delta,j).
\]
Therefore,
using~\eqref{that_could_cost} and~\eqref{somebodys_life}, the sum
of terms of~$G(z)$ for which~$q\nmid uj$ is bounded in absolute
value by
\[
\sum_{j}\sum_{u\in U(j)}\frac{c_j|\mathcal{I}(S'(j))|}{M(\delta,j)u}
=
\sum_{j}\frac{c_j|\mathcal{I}(S'(j))|}{M(\delta,j)}\prod_{p\in S(j)}\frac{1}{p^E(1-1/p)}.
\]

It remains to consider the terms of~$G(z)$ for which~$0<\ord_q(uj)<E$. In this case,~$\ord_q(uj)=\ord_q(j)>0$, so
$q\in S'(j)$.
Let~$S''(j)=S'(j)\setminus\{q\}$, and notice that we may write
\begin{eqnarray}
\nonumber
H_{S'(j)}(w)
& = &
H_{S''(j)}(w)-H_{S''(j)}(w^q)\\
\nonumber
& = &
\sum_{I\in\mathcal{I}(S''(j))}
d_I\left(
\frac{w^{k_I}}{1-w^{k_I}}
-\frac{w^{qk_I}}{1-w^{qk_I}}
\right)\\
\label{as_long_as_i_can}
& = &
\sum_{I\in\mathcal{I}(S''(j))}
d_I\left(
\frac{w^{k_I}+w^{2k_I}+\dots+w^{(q-1)k_I}}{1-w^{qk_I}}
\right),
\end{eqnarray}
where,~$q\nmid k_I$ for all~$I\in\mathcal{I}(S''(j))$.

Once again, for the terms of~$G(z)$ with~$0<\ord_q(uj)<E$, the
inner sum in~(\ref{that_could_cost}) is~$H_{S'(j)}(z^{uj'})$,
and using~(\ref{as_long_as_i_can}) we obtain the bound
\begin{equation}\label{but_im_gonna_bear_it}
|H_{S'(j)}(z^{uj'})|\leqslant
\sum_{I\in \mathcal{I}(S''(j))}
\frac{q-1}{|1-z^{uj'qk_I}|}.
\end{equation}
Furthermore, since
\[
\ord_q(uj'qk_I)=1+\ord_q(uj)\leqslant E,
\]
there is a constant~$M(\delta,j)>0$ such
that~$|1-z^{uj'qk_I}|>M(\delta,j)$. Just as before we can
use~(\ref{that_could_cost}) and~(\ref{but_im_gonna_bear_it}),
to obtain a bound for the remaining terms of~$G(z)$.
\qed\end{proof}

\begin{corollary}\label{with_no_other_man}
If~$S$ is finite and~$F_S$ is irrational, then the circle of
convergence of~$F_S$ is a natural boundary for the function.
\end{corollary}

\begin{proof}
Exactly as in the proof of Theorem~\ref{rationality_theorem},
and without loss of generality, we can assume that~$f_S(n)$ is
given by~\eqref{im_just_a_dirty_dog}. Then, as in
Example~\ref{vicky_example}, write
\[
F_S(z)=F(az)-F(bz),
\]
where~$F$ is given by
Theorem~\ref{distinguished_singularities}. Since~$a>|b|$ and
since~$F$ has the unit circle as a natural boundary, this shows
that~$F_S$ has the circle~$|z|=\frac{1}{a}$ as a natural
boundary. \qed\end{proof}

In the next section we will give an alternative proof of
Corollary~\ref{with_no_other_man} for higher-dimensional
solenoids based on the P\'olya--Carlson Theorem. However, this
method does not explicitly reveal why the natural boundary
occurs in the way that
Theorem~\ref{distinguished_singularities} does.

To conclude this section, we consider the occurrence of natural
boundaries for certain infinite sets of primes~$S$. We begin
with the following application of Fabry's gap Theorem (see
Segal~\cite[Sec.~6.4]{MR805682}).

\begin{lemma}\label{fabry_lemma}
Let~$a>1$ and let~$g(n)$ be an integer-valued sequence
satisfying~$g(n)\le a^n$ for all~$n\ge 1$.  Suppose that $g(n)$ does not satisfy a linear recurrence and suppose that for
every real number~$s>1$ there is a sequence of natural
numbers~$n_1<n_2<\cdots~$ with~$n_j/j\to \infty$ such
that~$g(n)<s^n$ for~$n\not \in \{n_1,n_2,\ldots \}$. Then there
is some~$R\in [\frac{1}{a},1]$ such that~$\sum_{n\ge 1} g(n)z^n$
admits the circle~$|z|=R$ as its natural boundary.
\end{lemma}

\begin{proof}
Let
\[
R^{-1}=\limsup_{n\rightarrow\infty} g(n)^{1/n}
\]
and notice that~$\frac{1}{a}\le R\le 1$. If~$R=1$, then the
P{\'o}lya--Carlson theorem immediately gives us the result.
Let~$s\in (1,\frac{1}{R})$, let~$N$ denote the set of natural
numbers~$n$ such that~$g(n) < s^n$, and write~$\{n_1,n_2,\ldots
\}$ for the complement of~$N$ in the natural numbers.  By
assumption~$n_j/j\to \infty$.  Let
\[
G(z)=\sum_{n\ge 1} g(n) z^n
\]
and
\[
G_N(z) = \sum_{n\in N} g(n)z^n.
\]
Then~$G_N$ has radius of convergence at least~$\frac{1}{s}$,
which is strictly greater than~$R$.  Hence the set of
singularities of~$G(z)-G_N(z)$ that lie on the circle of
radius~$R$ is identical to the set of singularities of~$G(z)$
on the circle of radius~$R$. However
\[
G(z)-G_N(z)=\sum_j g(n_j)z^{n_j}
\]
has the circle of radius~$R$ as its natural boundary by Fabry's
gap Theorem. \qed\end{proof}

\begin{lemma}\label{lem: s}
Let~$p$ be a rational prime such that~$|r|_p=1$. Then there is
an integer constant~$A$ depending only on~$r$ such that
\[
|r^n-1|_p^{-1}<\left(\frac{\log(n)}{n}+\frac{\ell_p\log(A)}{n}\right)^n,
\]
where~$\ell_p$ denotes the smallest natural number~$n$ for
which~$|r^n-1|_p<\frac{1}{2}$.
\end{lemma}

\begin{proof}
First note that
if~$p>2$,~$|r^n-1|_p<1/2\Longleftrightarrow|r^n-1|_p<1$,
so~$\ell_p=m_p$, where~$m_p$ is as in
Lemma~\ref{fundamental_evaluation_lemma}, but this is not the
case if~$p=2$. If~$n$ is not a multiple of~$\ell_p$,
then~$|r^n-1|_p^{-1}\le 2$, by
Lemma~\ref{fundamental_evaluation_lemma} and
Remark~\ref{fundamental_lemma_extension}, so the result is
immediate. For~$n$ a multiple of~$\ell_p$, write~$n=\ell_p p^e
k$ where~$p\nmid k$ and~$e\geqslant 0$. Then, by
Lemma~\ref{fundamental_evaluation_lemma} and
Remark~\ref{fundamental_lemma_extension},
\begin{equation*}\label{eq: mp}
|r^n-1|_p^{-1}
=
|r^{\ell_p p^e k}-1|_p^{-1}
\le
\ell_p |r^{\ell_p}-1|_p^{-1}p^e
\le
A^{\ell_p} p^e,
\end{equation*}
where~$A$ is the height of~$r$ (so~$A$ is an integer constant
depending only on~$r$ and is independent of~$p$). Furthermore,
\[
A^{\ell_p} p^e = \left(\frac{e\log(p)+\ell_p\log(A)}{n}\right)^n,
\]
giving the desired result since~$e\log(p)\le\log(n)$.
\qed\end{proof}

\begin{theorem}
Let~$S$ be a set of primes such
that~$\mathcal{P}(\mathbb{Q})\setminus S=\{p_1<p_2<\cdots \}$
is infinite and satisfies
\[
\frac{\log(p_{n+1})}{p_n}\longrightarrow\infty
\]
as~$n\to\infty$.
Then~$F_S$ admits some circle of radius~$R\le 1$ as its natural
boundary.
\end{theorem}

\begin{proof}
Let~$T=\mathcal{P}(\mathbb{Q})\setminus S=\{p_1<p _2<\dots\}$.
By our earlier assumptions on~$S$,~$|r|_p=1$ for all~$p\in T$.
Furthermore~$f_S(n)=|r^n-1|_T^{-1}$ by the Artin product
formula.

Let~$s>1$ and let~$C$ be a natural number with~$2\log(A)/C<
\log(s)/2$, where~$A$ is the constant appearing in
Lemma~\ref{lem: s}. For each~$k$, we let~$\ell_k$ denote the
smallest natural number such that~$|r^{\ell_k}-1|_{p_j}<1/2$.
Note that~$A^{\ell_k}>p_k$ and so~$\ell_k>\log(p_k)$.  On the
other hand,~$\ell_k\le p_k-1$, when~$p_k\neq 2$. Since
\[
p_k={\rm o}(\log(p_{k+1}))
\]
and~$\ell_{k+1}>\log(p_{k+1})$, we see
that~$\ell_{k+1}/\ell_k\to \infty$.

Assume that~$k$ is large enough to ensure
that~$\ell_i>C\ell_{i-1}$ for~$i\ge k$
and
\[
\ell_{k-1}>\max\{C\ell_1,\ldots ,C\ell_{k-2}\}.
\]
Let~$n$ be a natural number in~$\{\ell_{k-1},\ldots,\ell_k-1\}$
that is not in~$\{\ell_{k-1} j: j\le C\}$. Then we
have~$f_S(n)=\prod_{i<k} |r^n-1|_{p_i}^{-1}$, since~$p_i$
cannot divide~$f_S(n)$ for~$i\ge k$. Note that if~$n$ is not a
multiple of~$\ell_i$ then $|r^n-1|_{p_i}^{-1}\ge 1/2$; if~$n$
is a multiple of~$\ell_i$ then~$|r^n-1|_{p_i}^{-1} \le s_i^n$,
where~$s_i\le \log(n)/n+\ell_i\log(A)/n$, by Lemma~\ref{lem:
s}. Hence
\[
\frac{\log(f_S(n))}{n}
\le
\sum_{i<k} \left(\frac{\log(n)}{n}+\frac{\ell_i\log(A)}{n}\right)
\le
\frac{k\log(n)}{n}
+
\sum_{i<k} \frac{\ell_i \log(A)}{n}.
\]
Since~$\ell_i/\ell_{i-1}\to \infty$, we see that for~$n$
sufficiently large we must have
\[
\frac{\log(f_S(n))}{n}
<
\left\{
\begin{array}{ll}
k\log(n)/n + 2\ell_{k-1}\log(a)/n & \mbox{ if }\ell_{k-1}\mid n,\\
k\log(n)/n + {\rm o}(1) & \mbox{ if }\ell_{k-1}\nmid n.
\end{array}
\right.
\]
If~$n$ is a multiple of~$\ell_{k-1}$ then we must
have~$n>C\ell_{k-1}$ by construction, and hence in either case
we have
\[
\frac{\log(f_S(n))}{n} < \frac{k\log(n)}{n} + \frac{\log(s)}{2}
\]
for~$k$ sufficiently large. We claim that~$k={\rm
o}(n/\log(n))$. To see this, it is sufficient to show
that~$n={\rm o}(k\log(k))$.  But~$n\ge\ell_{k-1}$
and~$\ell_{k-1}>\log(p_{k-1})$; moreover, by
assumption,~$\log(\log(p_{k-1})))>p_{k-3}$ for~$k$ sufficiently
large and since~$p_{k-3}$ is necessarily greater than~$k$
for~$k$ large, we have
\[
n\ge\ell_{k-1}\ge \log(p_{k-1})>\exp(k)
\]
for~$k$ large, giving the claim. It follows that for~$k$
sufficiently large there are at most~$C$ values of~$n\in
\{\ell_{k-1},\ldots ,\ell_k-1\}$ for which we have $f_S(n)>s^n$.
Hence there is some constant~$B$ such that
\[
|\{n<\ell_{k+1}: f_S(n)>s^n\}| \le Ck+B.
\]
Let~$n_1<n_2<\cdots~$ be those natural numbers with the
property that~$f_S(n)>s^n$; that is,~$f_S(n)>s^n$ if and only
if~$n=n_i$ for some~$i$. Then, given a natural number~$i$,
there is some~$k$ such that $Ck+B<i \le C(k+1)+B$. By the
remarks above we have~$n_i>\ell_{k+1}$.  Thus
\[
\frac{n_i}{i}>\frac{\ell_{k+1}}{Ck+B+C}\to \infty
\]
as~$k\to \infty$, since we have shown that~$\ell_k\ge\exp(k)$
for~$k$ sufficiently large. The result now follows from Lemma
\ref{fabry_lemma}, taking $g(n)=f_S(n)$ and noting that $g(n)$ does not satisfy a non-trivial linear recurrence by Theorem \ref{rationality_theorem}. \qed\end{proof}

\section{Higher-dimensional groups}

A compact connected abelian group~$X$ of dimension~$d\geqslant
1$ has a Pontryagin dual group~$\widehat{X}$ that is a subgroup
of~$\mathbb{Q}^d$. For an automorphism~$\theta:X\rightarrow X$,
we use the following periodic point counting formula, taken
from~\cite[Th.~1.1]{MR2441142}. As before,~$\fix_\theta(n)$
denotes the number of points fixed by the
automorphism~$\theta^n$,~$\mathcal{P}(\mathbb{K})$ denotes the
set of places of the number field~$\mathbb{K}$, and for any set
of places~$S$ in~$\mathcal{P}(\mathbb{K})$, we
write~$|x|_S=\prod_{v\in S}|x|_v$.

\begin{proposition}\label{finite_places_pp_formula}
If~$\theta:X\to X$ is an ergodic automorphism of a finite
dimensional compact connected abelian group, then there exist
algebraic number fields~$\mathbb{K}_1,\dots,\mathbb{K}_k$, sets
of finite places~$T_j\subset \mathcal{P}(\mathbb{K}_j)$ and
elements~$\xi_j\in \mathbb{K}_j$, no one of which is a root of
unity for~$j=1,\dots,k$, such that
\begin{equation}\label{pp_eq}
\fix_\theta(n)=\prod_{j=1}^{k}
|\xi_j^n-1|_{T_j}^{-1}.
\end{equation}
\end{proposition}

The fields~$\mathbb{K}_j$ and sets of places~$T_j$ appearing
above depend only on~$X$ and~$\theta$, and are obtained by
considering~$\widehat{X}$ as a module over the Laurent
polynomial ring~$\laurent$, where the module structure is given
by identifying multiplication by~$t$ with the application of
the dual map~$\widehat{\theta}$ (a standard procedure for the
study of automorphisms of compact abelian
groups, see Schmidt~\cite{MR1345152} for an overview). The precise method
for obtaining the formula is constructive and is described
in~\cite[Sec.~4]{MR2441142}; it is useful to note from this
that~$\mathbb{K}_j=\mathbb{Q}(\xi_j)$,~$j=1,\dots,k$. As in the
one-dimensional case, applying the Artin product formula
to~(\ref{pp_eq}) gives
\begin{equation}\label{infinite_places_pp_formula}
\fix_\theta(n)=
\prod_{j=1}^{k}
|\xi_j^n-1|_{P_j^\infty\cup S_j},
\end{equation}
where~$P_j^\infty$ denotes the set of infinite places
of~$\mathbb{K}_j$ and~$S_j=\mathcal{P}(\mathbb{K}_j)\setminus
T_j$. It is also worth noting that~\cite[Rmk.~1]{MR2441142}
implies that~$|\xi_j|_v=1$ for all~$v\in T_j$,~$j=1,\dots,k$,
as~$\theta$ is an automorphism. The main result of this section
is the following.

\begin{theorem}\label{main_higher_dimension_theorem}
Under the assumptions of
Proposition~\ref{finite_places_pp_formula}, suppose that the
product in~\eqref{infinite_places_pp_formula} only involves
finitely many places and that~$|\xi_j|_v\neq 1$ for all~$v$
in~$P_j^\infty$ and~$j=1,\dots,k$. Then~$\zeta_\theta$ is
either rational or has a natural boundary at its circle of
convergence, and the latter occurs if and only if~$|\xi_j|_v=1$
for some~$v\in S_j$,~$1\le j\le k$.
\end{theorem}

The condition that~$|\xi_j|_v\neq 1$ for all~$v\in
P_j^\infty$,~$j=1,\dots,k$,  is equivalent to the statement
that none of the~$\xi_j$ have algebraic conjugates on the unit
circle. Hence, for example, the theorem applies if each~$\xi_j$
is a Pisot number. We remark that we do not believe the
avoidance of conjugates on the unit circle is necessary for the
dichotomy stated in the theorem, but it is essential in the
proof we provide. Nonetheless, we obtain the following
generalization of Corollary~\ref{with_no_other_man}.

\begin{corollary}
If the dimension of~$X$ is at most~$3$ and the product
in~\eqref{infinite_places_pp_formula} comprises finitely many
places, then~$\zeta_\theta$ is rational or has a natural
boundary at the circle of convergence.
\end{corollary}

\begin{proof}
If the dimension of~$X$ is at most 3 then each of the field
extensions~$\mathbb{K}_j|\mathbb{Q}$ has degree at most 3.
Therefore, if any algebraic conjugate of any~$\xi_j$ lies on
the unit circle then~$\xi_j$ must be a root of unity. However,
this is precluded by the hypothesis of
Proposition~\ref{finite_places_pp_formula}. \qed\end{proof}

For the proof of Theorem~\ref{main_higher_dimension_theorem} we
use the following.

\begin{lemma}\label{derivative_trick_lemma}
Let~$S$ be a finite list of places of algebraic number fields
and, for each~$v\in S$, let~$\xi_v$ be a non-unit root in the
appropriate number field such that~$|\xi_v|_v=1$. Then the
function
\[
F(z)=\sum_{n\geqslant 1}f(n)z^n,
\]
where~$f(n)=\prod_{v\in S}|\xi_v^n-1|_v$ for~$n\ge1$, has the unit circle as
a natural boundary.
\end{lemma}

\begin{proof}
We will use the notation of
Lemma~\ref{fundamental_evaluation_lemma}. First note that~$F$
has radius of convergence~$1$. For any function~$G$ that is
analytic inside the unit circle and which can be analytically
continued beyond it, both~$G'$ and~$z\longmapsto zG'(z)$ can also
be analytically continued beyond it.

Assume, for the purposes of a contradiction, that~$F$ has
analytic continuation beyond the unit circle.
Let~$\varrho=\sum_{v\in S}\varrho_v$ be the sum of all the
residue degrees, and repeat the process of differentiating then
multiplying by~$z$ precisely~$\varrho$ times, beginning with
the function~$F$ and finally obtaining the function~$\sum_{n\ge1}
n^\varrho f(n) z^n$, which by construction is analytic in a
region strictly containing the open unit disc (by our initial
observation). So too then is the function~$G$ defined
by
\[
G(z)=C\sum_{n\ge1}n^\varrho f(n) z^n,
\]
where~$C=\prod_{v\in S} C_v^{-1}$, the constant~$C_v$ being
given by Remark~\ref{fundamental_lemma_extension}. Note
that
\[
\limsup_{n\rightarrow\infty}(n^\varrho f(n))^{1/n}=1,
\]
so~$G$ also has radius of convergence~$1$.

We claim that~$Cn^\varrho f(n)$ is a positive integer for
all~$n\ge 1$. To see this, it is enough to notice
that~$C_v^{-1} n^{\varrho_v}|\xi^n-1|_v$ is a positive integer for
all~$n\ge 1$, which is shown by
Lemma~\ref{fundamental_evaluation_lemma}. Since~$G$ has integer
coefficients, radius of convergence~$1$ and analytic
continuation beyond the unit circle, it is a rational function
by the P\'olya--Carlson Theorem. Furthermore,~$H(z)=C \sum_{n\ge1}
n^\varrho z^n$ defines a rational function, and~$F$ is the
Hadamard quotient of~$G$ by~$H$. Since the sequence of
coefficients~$f=(f(n))$ is drawn from a finitely generated ring
over~$\mathbb{Z}$, the Hadamard quotient theorem implies
that~$F$ is also a rational function. Therefore,~$f$ is given
by a linear recurrence sequence.

Let~$p_v$ denote the characteristic of the residue
field~$\mathfrak{K}_v$ and let~$q$ be a rational prime coprime
to all~$p_v$ for each~$v\in S$. Proceeding as in the proof of
Theorem~\ref{rationality_theorem}, define $n(e)=q^e\prod_{v\in
S} m_vp_v^{D_v}$, where~$D_v$ is given
by~Lemma~\ref{fundamental_evaluation_lemma}
and~$e\in\mathbb{N}$. Then~$f(kn(e))=f(n(e))$ whenever~$k$ is
coprime to~$n(e)$. Hence, the sequence~$f$ takes on infinitely
many values infinitely often, which is not possible for a
linear recurrence sequence by the result of Myerson and van der
Poorten~\cite{MR1357486}. This contradiction means that the
assumption that~$F$ has analytic continuation beyond the unit
circle is untenable.\qed\end{proof}

\begin{proof}[of Theorem~\ref{main_higher_dimension_theorem}]
Let~$S_j'=\{v\in S_j:|\xi_j|_v\neq 1\}$, let
\[
f(n)=\prod_{j=1}^k|\xi_j^n-1|_{S_j\setminus S_j'},
\]
and let
\[
g(n)=\prod_{j=1}^k|\xi_j^n-1|_{P_j^\infty\cup S_j'}.
\]
So,~$\fix_\theta(n)=f(n)g(n)$ by~(\ref{infinite_places_pp_formula}).
By the ultrametric property
\[
g(n)=\prod_{j=1}^k
|\xi_j|^n_{S_j''}
|\xi_j^n-1|_{P_j^\infty},
\]
where~$S_j''=\{v\in S_j:|\xi_j|_v> 1\}$. Extending the method
of Smale~\cite{MR0228014} for toral automorphisms, we can
expand the product over infinite places using an appropriate
symmetric polynomial (see for example
\cite[Lem.~4.1]{MR2308145}) to obtain an expression of the form
\begin{equation}\label{all_of_my_future2}
g(n)=\sum_{I\in\mathcal{I}}d_I w_I^n,
\end{equation}
where~$\mathcal{I}$ is a finite indexing set,~$d_I\in\{-1,1\}$
and~$w_I\in\mathbb{C}$.

Moreover, since~$|\xi_j|_v\neq 1$ for all~$v\in
P_j^\infty$,~$j=1,\dots,k$, there is a dominant term~$w_J$ in
the expansion~\eqref{all_of_my_future2}, for which
\[
|w_J|
=
\prod_{j=1}^{k}
|\xi_j|_{S_j''}
\prod_{v\in P_j^\infty}
\max\{|\xi_j|_v,1\}
=
\prod_{j=1}^{k}
\prod_{v\in P_j^\infty\cup \mathcal{P}(\mathbb{K}_j)}
\max\{|\xi_j|_v,1\},
\]
and~$|w_J|>|w_I|$ for all~$I\neq J$ (note that~$\log|w_J|$ is
the topological entropy, as given by~\cite{MR961739}).
Furthermore, by~\eqref{all_of_my_future2},
\[
\zeta_\theta(z)
=
\exp\left(\sum_{I\in\mathcal{I}}d_I
\sum_{n\ge 1}
\frac{f(n)(w_Iz)^n}{n}\right).
\]
If~$S_j\setminus S_j'=\varnothing$ for all~$j=1,\dots,k$,
then~$f(n)\equiv 1$, and it follows immediately
that~$\zeta_\theta$ is rational.

Now suppose that~$S_j\setminus S_j'\neq\varnothing$ for
some~$j$. As noted in
Lemma~\ref{zeta_generating_function_link}, we need only exhibit
a natural boundary at the circle of convergence for
\[
\sum_{I\in\mathcal{I}}d_I \sum_{n\ge1} f(n)(w_Iz)^n
\]
to exibit one for~$\zeta_\theta(z)$.
Moreover,~$\limsup_{n\rightarrow\infty}f(n)^{1/n}=1$, so for
each~$I\in\mathcal{I}$, the series
\[
\sum_{n\ge1}f(n)(w_Iz)^n
\]
has radius of convergence~$|w_I|^{-1}$, and
since~$|w_J|^{-1}<|w_I|^{-1}$ for all~$I\neq J$, this means
that it suffices to show that the circle of
convergence~$|z|=|w_J|^{-1}$ is a natural boundary for~$\sum_{n\ge1}
f(n)(w_Iz)^n$. But this is the case precisely when~$\sum_{n\ge1}
f(n)z^n$ has the unit circle as a natural boundary, and this
has already been dealt with by
Lemma~\ref{derivative_trick_lemma}. \qed\end{proof}

We conclude with the following example.

\begin{example}
Suppose~$X$ is a two dimensional solenoid, so~$\widehat{X}\hookrightarrow\mathbb{Q}^2$.
Consider an automorphism~$\theta:X\rightarrow X$ dual to multiplication by the
matrix~$A=\left(
\begin{smallmatrix}
2 & 1\\
1 & 1
\end{smallmatrix}
\right)$
on~$\widehat{X}$. In the case~$X=\mathbb{T}^2$,~$\theta$ is the well-known
cat map which has dynamical zeta function
\begin{equation}\label{cat_map_zeta_function}
\zeta_\theta(z)=\frac{(1-z)^2}{(1-\xi z)(1-\eta z)},
\end{equation}
where~$\xi=(3+\sqrt{5})/2$ and~$\eta=\xi^{-1}=(3-\sqrt{5})/2$.
More generally, the method of~\cite[Sec.~4]{MR2441142} shows that there are
natural~$\mathbb{Z}[\xi^{\pm 1}]$-module
embeddings~$\mathbb{Z}[\xi^{\pm 1}~]\hookrightarrow\widehat{X}\hookrightarrow\mathbb{Q}(\xi)$,
where~$\widehat{X}$ is considered as a~$\mathbb{Z}[\xi^{\pm 1}]$-module by
identifying multiplication by~$A$ with multiplication by~$\xi$.
Dynamically, this means that the cat map on~$\mathbb{T}^2$ is always an algebraic factor
of~$(X,\theta)$. Furthermore, the places appearing in~\eqref{infinite_places_pp_formula}
are a subset of~$\mathcal{P}^\infty(\mathbb{Q}(\xi))\cup\mathcal{P}(\mathbb{Q}(\xi))$ and,
by expanding the product over infinite places, this formula simplifies to
\[
\fix_\theta(n)=(\xi^n+\eta^n-2)|\xi^n-1|_S,
\]
where~$S$ is the set of places~$v\in\mathcal{P}(\mathbb{Q}(\xi))$
for which~$|\cdot|_v$ is unbounded on~$\widehat{X}$ under the natural
embedding~$\widehat{X}\hookrightarrow\mathbb{Q}(\xi)$.
In the case~$S=\varnothing$,~$\zeta_\theta$ is given by~\eqref{cat_map_zeta_function}.
In any other case when~$S$ is finite,~$\zeta_\theta$ is shown to have a natural boundary
on the circle~$|z|=\eta$ by Theorem~\ref{main_higher_dimension_theorem}.
\end{example}


\providecommand{\bysame}{\leavevmode\hbox to3em{\hrulefill}\thinspace}

\end{document}